\documentclass[11pt,reqno]{amsart}

\usepackage{url}
\usepackage{amssymb}
\usepackage{amscd}
\usepackage{amsfonts}
\usepackage{amsmath}
\usepackage{eepic}
\usepackage{gastex}
\usepackage{amsthm}
\usepackage{amsgen}
\usepackage{eucal}
\usepackage[square]{natbib}

\newcommand{\sgp}{semi\-group}
\newcommand{\sgps}{semi\-groups}

\DeclareMathOperator{\al}{alph}
\DeclareMathOperator{\lop}{lop}

\theoremstyle{plain}
\newtheorem{theorem}{Theorem}
\newtheorem{lemma}[theorem]{Lemma}
\newtheorem{proposition}[theorem]{Proposition}

\theoremstyle{remark}

\title[Identities of the Kauffman Monoid $\mathcal{K}_3$]{Identities of the Kauffman Monoid $\mathcal{K}_3$}

\author{Yuzhu Chen, Xun Hu, N. V. Kitov, Yanfeng Luo, M. V. Volkov}
\address{(Yuzhu Chen, Xun Hu, Yanfeng Luo) Department of Mathematics and Statistics, Lanzhou University,
Lanzhou, Gansu, 730000, China} \email{luoyf@lzu.edu.cn}

\address{(Xun Hu) Department of Mathematics and Statistics, Chongqing Technology and Business University,
Chongqing, 400033, China}

\address{(N. V. Kitov, M. V. Volkov) Institute of Natural Sciences and Mathematics, Ural Federal University,
Lenina 51, 620000 Ekaterinburg, Russia} \email{n.v.kitov@urfu.ru, m.v.volkov@urfu.ru}

\begin{document}

\begin{abstract}
We give a transparent combinatorial characterization of the identities satisfied by the Kauffman monoid $\mathcal{K}_3$. Our characterization leads to a polynomial time algorithm to check whether a given identity holds in $\mathcal{K}_3$.
\end{abstract}

\maketitle

\section*{Introduction}

The present paper is a follow-up of the article by \citet{ACHLV15}. In particular, the object we deal with here (the Kauffman monoid $\mathcal{K}_3$) belongs to the family of monoids studied in that article. We reproduce here the definition of this family, closely following \citep{ACHLV15}.

\citet{TL71}, motivated by some graph-theoretical problems in
statistical mechanics, introduced a family of associative linear
algebras with 1 over the field $\mathbb{C}$. Given an integer
$n\ge 2$ and a scalar $\delta\in\mathbb{C}$, the Temperley--Lieb algebra
$\mathcal{TL}_n(\delta)$ has generators $h_1,\dots,h_{n-1}$ and relations
\begin{align}
&h_{i}h_{j}=h_{j}h_{i}    &&\text{if } |i-j|\ge 2,\ i,j=1,\dots,n-1;\label{eq:TL1}\\
&h_{i}h_{j}h_{i}=h_{i}    &&\text{if } |i-j|=1,\ i,j=1,\dots,n-1;\label{eq:TL2}\\
&h_{i}h_{i}=\delta h_{i}  &&\text{for each } i=1,\dots,n-1.\label{eq:TL3}
\end{align}
Since the relations \eqref{eq:TL1}--\eqref{eq:TL3} do not involve addition, the algebra $\mathcal{TL}_n(\delta)$ is spanned by its multiplicative submonoid generated by $h_1,\dots,h_{n-1}$.
This suggests introducing the monoid $\mathcal{K}_n$ with $n$ generators $c,h_1,\dots,h_{n-1}$ subject to the relations \eqref{eq:TL1}, \eqref{eq:TL2}, and the
relations
\begin{align}
&h_{i}h_{i}=ch_{i}=h_{i}c   &&\text{for each } i=1,\dots,n-1,\label{eq:TL4}
\end{align}
which both mimic \eqref{eq:TL3} and mean that $c$ behaves like the scalar $\delta$. The monoids $\mathcal{K}_n$ are called the \emph{Kauffman monoids}\footnote{The name comes from \citep{BDP02}; in the literature one also meets the name \emph{Temperley--Lieb--Kauffman monoids} \citep[see, e.g.,][]{BL05}. Kauffman himself used the term \emph{connection monoids}.} after \citet{Ka90} who independently invented these monoids as geometrical objects; see \citep[Section~1]{ACHLV15} for a geometric definition of the monoids $\mathcal{K}_n$.

Kauffman monoids play a role in knot theory, low-dimensional topology, topological quantum field theory, quantum groups, etc. As algebraic objects, these monoids belong to the family of so-called diagram or Brauer-type monoids that originally arose in representation theory \citep{Br37}. Various diagram monoids, including Kauffman ones, have gained much attention among semigroup theorists over the last two decades; see, e.g., \citep{Au12,Au14,ADV12,ACHLV15,DE17,DE18,Dea15,DEG17,Ea11a,Ea11b,Ea14a,Ea14b,EF12,EG17,Eea18,FL11,KMM06,KM06,KM07,LF06,MM07,Ma98,Ma02}.

In particular, the finite basis problem for the identities satisfied by Kauffman monoids has been solved by \citet{ACHLV15} who proved that, for each $n\ge 3$, the identities holding in the monoid $\mathcal{K}_n$ are not finitely based. The proof was based on a very `high-level' sufficient condition for the absence of a finite identity basis; if a semigroup $\mathcal{S}$ satisfies this condition, one can conclude that $\mathcal{S}$ admits no finite identity basis, without writing down any concrete identity holding in $\mathcal{S}$! Thus, no information about the identities of $\mathcal{K}_n$ for $n\ge 3$ can be extracted from the proofs in \citep{ACHLV15}, besides, of course, the mere fact that non-trivial identities in $\mathcal{K}_n$ do exist (since they have no finite basis).

As mentioned in  \citep{ACHLV15}, an alternative approach for the finite basis problem for $\mathcal{K}_3$ was independently developed by three of the authors of the present paper (Chen, Hu, and Luo). That approach relied on purely syntactic techniques and required, as an intermediate step, a combinatorial characterization of the identities satisfied by $\mathcal{K}_3$. Even though the characterization appeared to be of independent interest, it remained unpublished for two reasons: first, its initial, calculation-based proof was rather bulky; second, its main application, that is, the absence of a finite identity basis for $\mathcal{K}_3$, was subsumed by a much more general result in \citep{ACHLV15}. Now, with the inclusion of Kitov and Volkov in the team, we have mastered a short, calculation-free proof of the characterization and, besides that, we have found a new application: namely, we have shown that the characterization leads to a polynomial time algorithm to check whether a given identity holds in $\mathcal{K}_3$. The short proof and the new application make the content of the present paper.

The characterization of the identities of $\mathcal{K}_3$ and its algorithmic version are presented in Sections~\ref{sec:wire} and~\ref{sec:nfb} respectively.

\section{Identities of $\mathcal{K}_3$}
\label{sec:wire}

Recall that for a semigroup $\mathcal{S}$, the notation $\mathcal{S}^1$ stands for the least monoid containing $\mathcal{S}$, that is\footnote{Here and throughout expressions like $A:=B$ emphasize that $A$ is defined to be $B$.}, $\mathcal{S}^1:=\mathcal{S}$ if $\mathcal{S}$ has an identity element and $\mathcal{S}^1:=\mathcal{S}\cup\{1\}$ if $\mathcal{S}$ has no identity element; in the latter case the multiplication in $\mathcal{S}$ is extended to $\mathcal{S}^1$ in a unique way such that the fresh symbol $1$ becomes the identity element in $\mathcal{S}^1$. We adopt the following notational convention: if $s$ is an element of a semigroup $\mathcal{S}$, then $s^0$ stands for the identity element of $\mathcal{S}^1$.

We fix a countably infinite set $X$ which we refer to as an alphabet; elements of $X$ are referred to as \emph{letters}. The set $X^+$ of finite sequences of letters forms a semigroup under concatenation which is called  the \emph{free semigroup over the alphabet $X$}. Elements of $X^+$ are called \emph{words over $X$}. The monoid $X^*:=(X^+)^1$ is called the \emph{free monoid} over $X$; its identity element is referred to as the \emph{empty word}. We will often use the well-known universal property of the free monoid: if $\mathcal{M}$ is a monoid, any map $X\to\mathcal{M}$ can be uniquely extended to a homomorphism $X^*\to\mathcal{M}$ sending the empty word to the identity element of $\mathcal{M}$.

If $w=a_1\cdots a_\ell$ with $a_1,\dots,a_\ell\in X$ is a word from $X^+$, the number $\ell$ is called the \emph{length} of $w$, and $a_1$ and $a_\ell$ are said to be the \emph{first letter} and, respectively, the \emph{last letter} of $w$. The length of the empty word is 0, while the first and the last letter of the empty word are undefined.

We say that a word $v\in X^+$ \emph{occurs} in a word $w\in X^+$ if $w$ can be factorized as $w=u_1vu_2$ for some words $u_1,u_2\in X^*$. In this situation, the words $u_1$ and $u_2$ are referred to as the \emph{left context} and, respectively, the \emph{right context} of the occurrence of $v$. Clearly, it may happen that $v$ has several occurrences in $w$; we order these occurrences according to the lengths of their left contexts so that the \emph{first occurrence} is the one with the shortest left context, and so on.

For a word $w\in X^*$, we denote by $\al(w)$ the \emph{content} of $w$, that is, the set of all letters that occur in $w$. Observe that $w$ is empty if and only if $\al(w)$ is the empty set. If $Y\subseteq X$, we denote by $w_Y$ the word obtained from $w$ by removing all occurrences of the letters in $Y$. Then $w_Y$ is empty if and only if $\al(w)\subseteq Y$.

An \emph{identity} is an expression of the form $u\bumpeq v$ with $u,v\in X^*$. If $\mathcal{M}$ is a monoid, we say that the identity $u\bumpeq v$ \emph{holds} in $\mathcal{M}$ or, alternatively, that $\mathcal{M}$ \emph{satisfies} the identity $u\bumpeq v$ if every homomorphism $\varphi\colon X^*\to \mathcal{M}$ \emph{equalizes} $u$ and $v$, that is, $u\varphi=v\varphi$. Similarly, if $u,v\in X^+$ and $\mathcal{S}$ is a semigroup, we say that the identity $u\bumpeq v$ \emph{holds} in $\mathcal{S}$ or that $\mathcal{S}$ \emph{satisfies} the identity $u\bumpeq v$ if every homomorphism from $X^+$ into $\mathcal{S}$ equalizes $u$ and $v$.

The following fact is a part of semigroup folklore but we include its proof for the sake of completeness.

\begin{lemma}
\label{lem:monoid}
If $u,v\in X^*$ and the identity $u\bumpeq v$ holds in a monoid $\mathcal{M}$, then so does the identity $u_Y\bumpeq v_Y$ for each $Y\subseteq X$.
\end{lemma}

\begin{proof}
We have to check that an arbitrary homomorphism $\varphi\colon X^*\to\mathcal{M}$ equalizes $u_Y$ and $v_Y$. Consider the homomorphism $\varphi_Y\colon X^*\to\mathcal{M}$ that extends the following map $X\to\mathcal{M}$:
\[
x\mapsto\begin{cases}
            x\varphi & \text{if } x\notin Y,\\
            1 & \text{if } x\in Y.
        \end{cases}
\]
Then $w\varphi_Y=w_Y\varphi$ for every $w\in X^*$, whence $u_Y\varphi=u\varphi_Y=v\varphi_Y=v_Y\varphi$ since $\varphi_Y$ equalizes $u$ and $v$.
\end{proof}

We also need a normal form for the elements of the Kauffman monoid $\mathcal{K}_n$; this form was suggested by \citet{Jo83}. By the definition, the elements of $\mathcal{K}_n$ can be represented as words over the alphabet $\{c,h_1,\dots,h_{n-1}\}$. For all $a,b$ such that $1\le a<b\le n-1$, let $h_{[b,a]}:=h_bh_{b-1}\cdots h_{a+1}h_a$; for the sake of uniformity, we also let $h_{[a,a]}:=h_a$. A word from $\{c,h_1,\dots,h_{n-1}\}^*$ is said to be in \emph{Jones's normal form} if it is either of the form $c^{\ell}h_{[b_1, a_1]}\cdots h_{[b_k,a_k]}$ for some $\ell\ge0$ and some $a_1 < \dots <a_k$, $b_1 < \dots < b_k$, or of the form $c^{\ell}$ for some $\ell\ge0$. The proofs of the next statement can be found in \citep{BDP02} and \citep{BL05}.

\begin{lemma}
\label{lem:jones}
Every element of the Kauffman monoid $\mathcal{K}_{n}$ has a unique representation as a word in Jones's normal form over $\{c,h_1,\dots,h_{n-1}\}$.
\end{lemma}

\begin{theorem}
\label{thm:description}
An identity $w\bumpeq w'$ holds in the Kauffman monoid $\mathcal{K}_3$ if and only if $\al(w)=\al(w')$ and, for each $Y\subset\al(w)$, the words $u:=w_Y$ and $u':=w'_Y$ satisfy the following three conditions:
\begin{itemize}
\item[(a)] the first letter of $u$ is the same as the first letter of $u'$;
\item[(b)] the last letter of $u$ is the same as the last letter of $u'$;
\item[(c)] for each word of length $2$, the number of its occurrences in $u$ is the same as the number of its occurrences in $u'$.
\end{itemize}
\end{theorem}

\begin{proof}
We start with a closer look at the monoid $\mathcal{K}_3$. Specializing the definition of the Kauffman monoids given in the introduction, one gets the following monoid presentation for $\mathcal{K}_3$:
\[
\mathcal{K}_3=\left\langle h_1,h_2,c\quad \begin{tabular}{|@{\quad}c} $h_1h_2h_1=h_1,\ h_2h_1h_2=h_2,$\\[.1ex] $h_1^2=ch_1=h_1c,\  h_2^2=ch_2=h_2c$\end{tabular}\right\rangle.
\]
Lemma~\ref{lem:jones} readily implies that every element in $\mathcal{K}_3$ is equal to a unique element of one of the following 5 sets:
\begin{align*}
  C &:=\{c^k\mid k=0,1,\dots\},&&\\
  H_{11} &:=\{c^\ell h_1\mid \ell=0,1,\dots\}, &
  H_{12} &:=\{c^m h_1h_2\mid m=0,1,\dots\},\\
  H_{21} &:=\{c^n h_2h_1\mid n=0,1,\dots\}, &
  H_{22} &:=\{c^r h_2\mid r=0,1,\dots\}.
\end{align*}

We turn to the proof of the `only if' part of our theorem. Let $w\bumpeq w'$ be an arbitrary identity that holds in $\mathcal{K}_3$. Given a letter $x_0\in X$, consider the homomorphism $\chi_0\colon X^*\to\mathcal{K}_3$
that extends the following map $X\to\mathcal{M}$:
\[
x\mapsto\begin{cases}
            c & \text{if } x=x_0,\\
            1 & \text{if } x\ne x_0.
         \end{cases}
\]
Then $w\chi_0=c^t$, where $t$ is the number of occurrences of $x_0$ in $w$, and similarly, $w'\chi_0=c^{t'}$, where $t'$ is the number of occurrences of $x_0$ in $w'$. Since $\chi_0$ must equalize $w$ and $w'$, we conclude that $t=t'$; in particular, $x_0$ occurs in $w$ if and only if it occurs in $w'$. Thus, $\al(w)=\al(w')$. If the word $w$ is empty, $\al(w)=\varnothing$ has no proper subsets and nothing remains to prove. Therefore, for the rest of the proof of the `only if' part, we assume that neither $w$ nor $w'$ is empty.

Let $\mathcal{S}_2$ stand for the semigroup presented by $\langle e,f \mid e^2=e,\ f^2=f\rangle$, that is, $\mathcal{S}_2$ is the free product of two trivial semigroups. Clearly, in $\mathcal{S}_2$ each element is uniquely represented as an alternating product of the generators $e$ and $f$. Hence, $\mathcal{S}_2$ is a disjoint union of the following 4 sets:
\begin{align*}
&\{(ef)^\ell\mid \ell=1,2,\dots\},&& \{(fe)^nf\mid n=0,1,\dots\},\\
&\{(ef)^me\mid m=0,1,\dots\},&&\{(fe)^r\mid r=1,2,\dots\}.
\end{align*}
We define  a map $\psi\colon\mathcal{S}_2\to\mathcal{K}_3$ as follows:
\begin{align*}
(ef)^\ell &\mapsto c^{2\ell-1}h_1&&\text{for each $\ell>0$},\\
(ef)^me   &\mapsto c^{2m}h_1h_2  &&\text{for each $m\ge0$},\\
(fe)^nf   &\mapsto c^{2n}h_2h_1  &&\text{for each $n\ge0$},\\
(fe)^r    &\mapsto c^{2r-1}h_2   &&\text{for each $r>0$}.
\end{align*}
Clearly, $\psi$ is 1-1, and a straightforward verification shows that $\psi$ is a homomorphism. Indeed, it suffices to compare Table~\ref{tb:S2}, which shows how typical elements of the semigroup $\mathcal{S}_2$ multiply, and Table~\ref{tb:K3}, which shows how the images of these elements under $\psi$ multiply.
\begin{table}[h]
\caption{Multiplication in $\mathcal{S}_2$}\label{tb:S2}
$\begin{array}{c|cccc}
         & (ef)^{\ell'} & (ef)^{m'}e & (fe)^{n'}f & (fe)^{r'} \\
   \hline
  (ef)^\ell & (ef)^{\ell+\ell'} & (ef)^{\ell+m'}e & (ef)^{\ell+n'} & (ef)^{\ell+r'-1}e\rule{0pt}{14pt} \\
  (ef)^me & (ef)^{m+\ell'} & (ef)^{m+m'}e & (ef)^{m+n'+1} & (ef)^{m+r'}e \rule{0pt}{14pt}\\
  (fe)^nf & (fe)^{n+\ell'}f & (fe)^{n+m'+1} & (fe)^{n+n'}f & (fe)^{n+r'} \rule{0pt}{14pt}\\
  (fe)^r & (fe)^{r+\ell'-1}f & (fe)^{r+m'} & (fe)^{r+n'}f & (fe)^{r+r'}\rule{0pt}{14pt}
\end{array}$
\end{table}

\begin{table}[h]
\caption{Multiplication in $\mathcal{S}_2\psi$}\label{tb:K3}
$\begin{array}{c|cccc}
                 & c^{2\ell'-1}h_1 & c^{2m'}h_1h_2 & c^{2n'}h_2h_1 & c^{2r'-1}h_2 \\
  \hline
  c^{2\ell-1}h_1 & c^{2(\ell+\ell')-1}h_1 & c^{2(\ell+m')}h_1h_2 & c^{2(\ell+n')-1}h_1 & c^{2(\ell+r'-1)}h_1h_2\rule{0pt}{14pt} \\
  c^{2m}h_1h_2 & c^{2(m+\ell')-1}h_1 & c^{2(m+m')}h_1h_2 & c^{2(m+n'+1)-1}h_1 & c^{2(m+r')}h_1h_2\rule{0pt}{14pt} \\
  c^{2n}h_2h_1 & c^{2(n+\ell')}h_2h_1 & c^{2(n+m'+1)-1}h_2 & c^{2(n+n')}h_2h_1 & c^{2(n+r')-1}h_2\rule{0pt}{14pt} \\
  c^{2r-1}h_2 & c^{2(r+\ell'-1)}h_2h_1 & c^{2(r+m')-1}h_2 & c^{2(r+n')}h_2h_1 & c^{2(r+r')-1}h_2\rule{0pt}{14pt}
\end{array}$
\end{table}

Thus, $\mathcal{S}_2$ is isomorphic to a subsemigroup in $\mathcal{K}_3$, whence $\mathcal{S}_2$ satisfies every identity $u\bumpeq v$ with $u,v\in X^+$ that holds in $\mathcal{K}_3$. By Lemma~\ref{lem:monoid}, for each proper subset $Y\subset\al(w)$, the identity $u\bumpeq u'$, where $u:=w_Y\in X^+$ and $u':=w'_Y\in X^+$, holds in $\mathcal{K}_3$. We conclude that $u\bumpeq u'$ holds in $\mathcal{S}_2$ as well, and by \citep[Theorem~3]{SV17}, the words $u$ and $u'$ satisfy conditions (a)--(c). This completes the proof of the `only if' part of the theorem.

For the `if' part, consider any identity $w\bumpeq w'$ satisfying the conditions of our theorem. If $\al(w)=\varnothing$, then the condition $\al(w)=\al(w')$ implies that both $w$ and $w'$ are empty words, and the identity $w\bumpeq w'$ holds in every monoid. Thus, we may assume that $\al(w)\ne\varnothing$. Take an arbitrary letter $x\in\al(w)$ and let $Y:=\al(w)\setminus\{x\}$. Then the words $u:=w_Y$ and $u':=w'_Y$ are certain powers of the letter $x$, namely, $u=x^t$ and $u'=x^{t'}$, where $t$ is the number of occurrences of $x$ in $w$ and $t'$ is the number of occurrences of $x$ in $w'$. Clearly, the word $x^2$ occurs $t-1$ times in the word $x^t$ and $t'-1$ times in the word $x^{t'}$, and since $u$ and $u'$ must satisfy the condition (c), we conclude that $t-1=t'-1$, whence $t=t'$.

We have to check that an arbitrary homomorphism $\varphi\colon X^*\to\mathcal{K}_3$ equalizes $w$ and $w'$. Recall that $\mathcal{K}_3$ is the disjoint union of the set $C$, which is a submonoid in $\mathcal{K}_3$, and the set $H:=\mathcal{K}_3\setminus C=H_{11}\cup H_{12}\cup H_{21}\cup H_{22}$, which is the ideal of $\mathcal{K}_3$ generated by $h_1$ and $h_2$. Let $Y:=\{y\in\al(w)\mid y\varphi\in C\}$. For each $y\in Y$, let $t_y$ stand for the number of occurrences of $y$ in $w$ (which, as shown in the preceding paragraph, is equal to the number of occurrences of $y$ in $w'$), and let $k_y\in\{0,1,\dots\}$ be such that $y\varphi=c^{k_y}$.
We denote the sum $\sum_{y\in Y}t_yk_y$ by $N_Y$. If $Y=\al(w)$, we have $w\varphi=c^{N_Y}=w'\varphi$, and we are done.

Consider the situation where $Y\subset\al(w)$. Using the fact that the generator $c$ commutes with the generators $h_1,h_2$, we can represent $w\varphi$ and $w'\varphi$ as $c^{N_Y}w_Y\varphi$ and $c^{N_Y}w'_Y\varphi$ respectively. Therefore it remains to verify that $w_Y\varphi=w'_Y\varphi$, and for this, it suffices to show that the identity $u\bumpeq u'$ with $u:=w_Y$ and $u':=w'_Y$ holds in the semigroup $H$.

We prove that $H$ satisfies  $u\bumpeq u'$, using the Rees matrix construction (cf. \cite[Chapter~3]{CP61}). Let $\mathbb{Z}$ stand for the additive group of integers and let $\Delta:=\begin{pmatrix}1&0\\0&1\end{pmatrix}$ be the identity $2\times2$-matrix over $\mathbb{Z}$. It is convenient for us to represent the matrix using Kronecker's delta notation so that $\Delta=\begin{pmatrix}\delta_{11}&\delta_{12}\\\delta_{21}&\delta_{22}\end{pmatrix}$. Denote by $\mathrm{M}(\mathbb{Z};\Delta)$ the set of triples
\[
\{(\eta,k,\lambda)\mid \eta,\lambda\in\{1,2\},\ k\in\mathbb{Z}\},
\]
endowed with the multiplication
\[
(\eta,k,\lambda)(\iota,\ell,\mu):=(\eta,k+\delta_{\lambda\,\iota}+\ell,\mu).
\]
The semigroup $\mathrm{M}(\mathbb{Z};\Delta)$ is an instance of the family of the \emph{Rees matrix \sgps} over $\mathbb{Z}$.

Define a map $\xi\colon H\to\mathrm{M}(\mathbb{Z};\Delta)$ as follows:
\begin{align*}
c^\ell h_1&\mapsto (1,\ell,1)&&\text{for each $\ell\ge0$},\\
c^m h_1h_2&\mapsto (1,m,2)&&\text{for each $m\ge0$},\\
c^n h_2h_1&\mapsto (2,n,1)&&\text{for each $n\ge0$},\\
c^r h_2   &\mapsto (2,r,2)&&\text{for each $r\ge0$}.
\end{align*}
Obviously, $\xi$ is 1-1, and one can readily verify that $\xi$ is a homomorphism. Thus, $H$ is isomorphic to a subsemigroup in $\mathrm{M}(\mathbb{Z};\Delta)$.  It is known (see, e.g., \citet[Theorem~9]{KR79}) and easy to verify that every identity $u\bumpeq u'$ with $u$ and $u'$ satisfying (a)--(c) holds in each Rees matrix \sgp\ over an abelian group. Hence, every such identity holds in $\mathrm{M}(\mathbb{Z};\Delta)$, and thus, in $H$. This completes the proof of the `if' part of the theorem.
\end{proof}

\section{Recognizing identities of $\mathcal{K}_3$ in polynomial time}
\label{sec:nfb}

Given a semigroup $\mathcal{S}$, its \emph{identity checking problem}\footnote{Also called the `\emph{term equivalence problem}' in the literature.} is a combinatorial decision problem whose instance is an arbitrary pair $(w,w')$ of words; the answer to the instance $(w,w')$ of the problem is `YES' or `NO' depending on whether or not the identity $w\bumpeq w'$ holds in $\mathcal{S}$. For a finite semigroup, the identity checking problem is always decidable, and moreover, belongs to the complexity class $\mathsf{coNP}$: if for some pair $(w,w')$ of words that together involve $m$ letters, the identity $w\bumpeq w'$ fails in the semigroup $\mathcal{S}$, then a nondeterministic polynomial algorithm can guess an $m$-tuple of elements in $\mathcal{S}$ witnessing the failure and then confirm the guess by computing the values of the words $w$ and $w'$ at this $m$-tuple. There exist many examples of finite semigroups whose identity checking problem is $\mathsf{coNP}$-complete; see, e.g., \citep{AVG09,HLMS,JM06,Ki04,Kl09,Kl12,PV06,Se05,SS06} and the references therein. However, the task of classifying finite semigroups according to the computational complexity of identity checking appears to be far from being feasible as it is not yet accomplished even in the case of finite groups.

For infinite semigroups, results on the identity checking problem are sparse. The reason for this is that infinite semigroups usually arise in mathematics as \sgps\ of transformations of an infinite set, or \sgps\ of relations on an infinite domain, or \sgps\ of matrices over an infinite ring, and as a rule all these \sgps\ are `too big' to satisfy any nontrivial identity. If, however, an infinite semigroup satisfies a nontrivial identity, its identity checking problem may constitute a challenge: \citet{Mu68} had constructed an infinite semigroup with undecidable identity checking problem. On the `positive' side, we mention a recent result by \citet{DJK18} who have shown that checking identities in the famous bicyclic monoid $\mathcal{B}:=\langle a,b\mid ba=1\rangle$ can be done in polynomial time via rather a non-trivial algorithm based on linear programming.

Observe that even though Theorem~\ref{thm:description} gives an algorithm to verify whether or not a given identity $w\bumpeq w'$ holds in the Kauffman monoid $\mathcal{K}_3$, the algorithm is not polynomial in the number of letters occurring in the words $w$ and $w'$ because one has to check conditions (a)--(c) for every proper subset of the set $\al(w)$. We will `unfold' this algorithm so that the unfolded version admits a polynomial-time implementation; our approach is inspired by a method developed by \citet{SS06} for checking identities in certain finite semigroups.

Given a word $w\in X^+$, its \emph{first} (\emph{last}) \emph{occurrence word} is obtained from $w$ by retaining only the first (respectively, the last) occurrence of each letter that occurs in $w$. A \emph{jump} is a triple $(x,G,y)$, where $x$ and $y$ are (not necessarily distinct) letters and $G$ is a (possibly empty) set of letters that contains neither $x$ nor $y$. The jump $(x,G,y)$ \emph{occurs} in a word $w$ if $w$ can be factorized as $w=v_1xv_2yv_3$ where $v_1,v_2,v_3\in X^*$ and $G=\al(v_2)$. For instance, each  of the jumps $(x,\{y,z\},x)$ and $(y,\varnothing,y)$ occurs twice in the word $xy^2zxzy^2x$, while each of the jumps $(x,\{y\},z)$ and $(z,\{y\},x)$ occurs just once.

The following result is in fact a reformulation of Theorem~\ref{thm:description} in a form amenable for an algorithmic analysis.
\begin{theorem}
\label{thm:jump}
An identity $w\bumpeq w'$ holds in the monoid $\mathcal{K}_3$ if and only if either both $w$ and $w'$ are empty or $w$ and $w'$ have the same first occurrence and the same last occurrence words, and every jump occurs the same number of times in $w$ and $w'$.
\end{theorem}

\begin{proof}
For the `only if' claim, we use the `only if' part of Theorem~\ref{thm:description}. In view of the latter, $\al(w)=\al(w')$, whence $w$ is empty whenever $w'$ is, and vice verse.
So we may assume that $w,w'\in X^+$. Since $w$ and $w'$ satisfy condition (a) of Theorem~\ref{thm:description}, they start with the same letter, say, $x_1$. If $\al(w)=\{x_1\}$, the first occurrence word of both $w$ and $w'$ is just $x_1$, and we are done. Otherwise $\{x_1\}$ is a proper subset of $\al(w)$, and therefore, condition (a) must be satisfied by the words $w_{\{x_1\}}$ and $w'_{\{x_1\}}$. Hence the first letter of $w_{\{x_1\}}$ is the same as the first letter of $w'_{\{x_1\}}$; let us denote this common letter by $x_2$. Observe that $x_2\ne x_1$ since $x_1$ does not occur in $w_{\{x_1\}}$ by the very definition of this word. If $\al(w)=\{x_1,x_2\}$, the first occurrence word of both $w$ and $w'$ is  $x_1x_2$, and we are done again. Otherwise $\{x_1,x_2\}$ is a proper subset of $\al(w)$, and we can repeat the argument until we exhaust the set $\al(w)$. At the $i$-th step of the procedure, we append the common first letter $x_i$ of the words $w_{\{x_1,\dots,x_{i-1}\}}$ and $w'_{\{x_1,\dots,x_{i-1}\}}$ to the already constructed word $x_1\cdots x_{i-1}$;
observe that $x_i\notin\{x_1,\dots,x_{i-1}\}$ by the definition of the word $w_{\{x_1,\dots,x_{i-1}\}}$. Clearly, the word we get at the end of the procedure is the common first occurrence word of $w$ and $w'$. In the dual way, we deduce that $w$ and $w'$ have the same last occurrence word.

It remains to show that an arbitrary jump $(x,G,y)$ occurs the same number of times in $w$ and $w'$. We fix the letters $x$ and $y$ and induct on the cardinality of $G$. If this cardinality is $0$, that is, $G=\varnothing$, each occurrence of the jump $(x,\varnothing,y)$ in a word are nothing but an occurrence of $xy$ in this word. Since $w$ and $w'$ satisfy condition (c) of Theorem~\ref{thm:description}, the word $xy$ must occur the same number of times in $w$ and $w'$, and so does the jump $(x,\varnothing,y)$.

The induction step relies on the following observation, which will be useful also in the proof of the `if' claim.
\begin{lemma}
\label{lem:jump}
Let $x$ and $y$ be (not necessarily distinct) letters, $v\in X^+$ a word, and $G\subseteq\al(v)$ a set of letters that includes neither $x$ nor $y$. The factor $xy$ occurs in the word $v_G$ as many times as jumps of the form $(x,H,y)$, where $H$ runs over the set of all subsets of $G$, occur in the word $v$.
\end{lemma}

\begin{proof}
For $xy$ to occur in $v_G$, the word $v$ should contain factors of the form $xsy$  where $\al(s)\subseteq G$ so that the `streak' $s$ disappears when the letters from $G$ get removed from $v$. In terms of jumps, this means that the occurrences of $xy$ in the word $v_G$ are in a 1-1 correspondence with the occurrences of jumps of the form $(x,H,y)$ with $H\subseteq G$ in the word $v$.
\end{proof}

Now consider a jump $(x,G,y)$ with $G\ne\varnothing$. Of course, we may assume that $x,y\in\al(w)$ and $G\subseteq\al(w)$. Then $G$ is a proper subset of $\al(w)$ since $x\notin G$. Consider the words $u:=w_G$ and $u':=w'_G$. They satisfy condition (c) of Theorem~\ref{thm:description}. Hence, if $m$ and $m'$ denote the numbers of occurrences of the word $xy$ in $u$ and respectively $u'$, we have $m=m'$. For any subset $H\subseteq G$, let $n_H$ and $n'_H$ stand for the numbers of occurrences of the jump $(x,H,y)$ in $w$ and respectively $w'$. By Lemma~\ref{lem:jump} we have
\begin{equation}
\label{eq:jump}
m=\sum_{H\subseteq G}n_H=n_G+ \sum_{H\subset G}n_H\quad\text{and}\quad m'=\sum_{H\subseteq G}n'_H=n'_G+ \sum_{H\subset G}n'_H.
\end{equation}
We have $m=m'$ and, by the induction assumption, $n_H=n'_H$ for each proper subset $H$ of $G$. Hence the equalities \eqref{eq:jump} imply that $n_G$=$n'_G$, as required. This completes the proof of the `only if' claim.

For the `if' claim, consider any words $w$ and $w'$ satisfying the conditions of our theorem. If both $w$ and $w'$ are empty, the identity $w\bumpeq w'$ holds in every monoid. Thus, we may assume that $\al(w)\ne\varnothing$. Take an arbitrary proper subset $G$ of $\al(w)$. We aim to show that the words $u:=w_G$ and $u':=w'_G$ satisfy conditions (a)--(c) of  Theorem~\ref{thm:description}; our claim then follows from the `if' part of the latter theorem.

Let $v$ be the first occurrence word of both $w$ and $w'$. Then it is easy to see that the word $v_G$ is the first occurrence word of both $u$ and $u'$. In particular, the first letter of $v_G$ occurs as the first letter in both $u$ and $u'$. Thus, $u$ and $u'$ satisfy condition (a). In the dual way, we obtain that $u$ and $u'$ satisfy condition (b).

In order to verify condition (c), take an arbitrary word $xy$ of length 2, where $x$ and $y$ are (not necessarily distinct) letters, and let $G\subseteq\al(v)$ be a set of letters that includes neither $x$ nor $y$. Re-using the notation $m,m',n_H,n'_H$ introduced in the last paragraph of the proof of the `only if' claim and applying Lemma~\ref{lem:jump}, we get
\[
m=\sum_{H\subseteq G}n_H\quad\text{and}\quad m'=\sum_{H\subseteq G}n'_H.
\]
Since $n_H=n'_H$ for each $H$, we conclude that $m=m'$, thus completing the proof of the `if' claim.
\end{proof}

It remains to show that, given an identity $w\bumpeq w'$, one can check whether or not the words $w$ and $w'$ satisfy conditions of Theorem~\ref{thm:jump} in polynomial of the sum of the lengths of $w$ and $w'$ time. For this, it suffices to exhibit algorithms that, given a word $v\in X^+$ of length $n$, find its first occurrence word, its last occurrence word, and its jumps with their multiplicities in polynomial in $n$ time. In fact, the first two algorithms require only $O(kn)$ time, where $k$ is the number of letters in $\al(v)$, and the third algorithm requires $O(kn\log(kn))$ time.

The algorithms for constructing the first and last occurrence words are pretty straightforward. For the first occurrence word, we initialize $\overrightarrow{v}$ as the empty word and then scan the input word $v$ letter-by-letter from left to right. Each time when we read a letter of $v$, we check whether the letter occurs in $\overrightarrow{v}$ and if it does not, we append the letter to $\overrightarrow{v}$. Then we pass to the next letter if it exists or stop if the current letter is the last letter of~$v$. Clearly, at the end of the process, $\overrightarrow{v}$ becomes the first occurrence word of~$v$. The algorithm, which we call FOW, makes $n$ steps and on each step it operates with the word $\overrightarrow{v}$ whose length does not exceed $k$. Hence, the time spent by FOW is linear in $kn$.

For the last occurrence word, we could apply FOW to the mirror image of the input and return the mirror image of the output of FOW. Alternatively, we suggest the following algorithm, which like FOW operates in the online manner, that is, processes its input word $v$ letter-by-letter from left to right. We initialize $\overleftarrow{v}$ as the empty word. Each time as a letter of $v$ is read, we check whether the letter occurs in $\overleftarrow{v}$. If it does, it occurs in $\overleftarrow{v}$ exactly once and we remove the occurrence from  $\overleftarrow{v}$. Then we append the current letter to $\overleftarrow{v}$ and pass to the next letter if it exists or stop if we have reached the last letter of $v$. At the end of the process, $\overleftarrow{v}$ becomes the last occurrence word of $v$, and again, the working time of the algorithm is linear in $kn$.

The algorithm that constructs the multiset of all jumps of $v$ is slightly more involved. We initialize $J$ as the empty multiset; besides that, for each letter $x\in\al(v)$, we introduce an integer variable denoted $\lop(x)$ (the \emph{last observed position of} $x$) and initialise it as 0. For each positive integer $i\le n$, we denote by $v[i]$ the letter in the $i$-th position of the input word $v$. For integers $i,j\le n$, we let
\[
v[i,j]:=\begin{cases}
v[i]\cdots v[j] &\text{if }\ i\le j,\\
\text{the empty word} & \text{if }\ i>j.
\end{cases}
\]

Our algorithm scans $v$ letter-by-letter from left to right. Suppose that the current position is $i$ and $v[i]=y$. For each letter $x\in\al(v)$ such that $\lop(x)>0$, we check if $\lop(y)\le\lop(x)$. If the inequality holds, then neither $x$ nor $y$ occurs in the factor $v[\lop(x)+1,i-1]$ of $v$ and we add the jump $(x,G,y)$ with $G:= \al(v[\lop(x)+1,i-1])$ to the mulitiset~$J$. (Recall that adding an element $e$ to a multiset $M$ means including $e$ in $M$ with multiplicity 1 if $e$ has not yet appeared in $M$ or increasing the multiplicity of $e$ in $M$ by 1 if $e$ has already appeared in $M$. By storing $M$ as an appropriate data structure, say, a self-balancing binary search tree, one can perform each such operation in $O(\log|M|)$ time. See \citet{St15} for a description of advanced techniques for handling multisets.) Then we update the variable $\lop(y)$ by assigning value $i$ to it and either stop if $i=n$ or pass to the position $i+1$ if $i<n$. Thus, the algorithm makes $n$ steps, at each step at most $k$ jumps are added to $J$, and the time needed for adding of each jump is bounded by $O(\log(kn))$. Hence the overall time spent is $O(kn\log(kn))$.

The following table demonstrates how the algorithm runs on the word $v=x^3yxyz^4xyz$. We have lowered the entries in the columns containing the values of the variables $\lop(x)$, $\lop(y)$, and $\lop(z)$ in order to stress that every step of the algorithm consists of two phases. Namely, when processing the letter $v[i]$, we first add jumps to the multiset  $J$ using the values of $\lop(x)$, $\lop(y)$, and $\lop(z)$ inherited from the previous step, and only after that we update one of these values.
\[
\begin{array}{c|c|c|c|c|c}
  i & v[i] & \lop(x) & \lop(y) & \lop(z) & \text{Jumps added to $J$}\\
  \hline
  1 & x & \raisebox{6pt}{0} & \raisebox{6pt}{0} & \raisebox{6pt}{0} & -\rule{0pt}{16pt} \\
  2 & x & \raisebox{6pt}{1} & \raisebox{6pt}{0} & \raisebox{6pt}{0} & (x, \varnothing, x) \\
  3 & x & \raisebox{6pt}{2} & \raisebox{6pt}{0} & \raisebox{6pt}{0} & (x, \varnothing, x) \\
  4 & y & \raisebox{6pt}{3} & \raisebox{6pt}{0} & \raisebox{6pt}{0} & (x, \varnothing, y) \\
  5 & x & \raisebox{6pt}{3} & \raisebox{6pt}{4} & \raisebox{6pt}{0} & (x, \{y\}, x),\, (y, \varnothing, x)\\
  6 & y & \raisebox{6pt}{5} & \raisebox{6pt}{4} & \raisebox{6pt}{0} & (y, \{x\}, y),\, (x, \varnothing, y) \\
  7 & z & \raisebox{6pt}{5} & \raisebox{6pt}{6} & \raisebox{6pt}{0} & (x, \{y\}, z),\, (y, \varnothing, z) \\
  8 & z & \raisebox{6pt}{5} & \raisebox{6pt}{6} & \raisebox{6pt}{7} & (z, \varnothing, z) \\
  9 & z & \raisebox{6pt}{5} & \raisebox{6pt}{6} & \raisebox{6pt}{8} & (z, \varnothing, z) \\
  10 & z & \raisebox{6pt}{5} & \raisebox{6pt}{6} & \raisebox{6pt}{9} & (z, \varnothing, z) \\
  11 & x & \raisebox{6pt}{5} & \raisebox{6pt}{6} & \raisebox{6pt}{10} & (x, \{y, z\}, x),\, (y, \{z\}, x),\, (z, \varnothing, x)\\
  12 & y & \raisebox{6pt}{11} & \raisebox{6pt}{6} & \raisebox{6pt}{10} & (y, \{z, x\}, y),\, (z, \{x\}, y),\, (x, \varnothing, y)\\
  13 & z & \raisebox{6pt}{11} & \raisebox{6pt}{12} & \raisebox{6pt}{10} & (z, \{x, y\}, z),\, (x, \{y\}, z),\, (y, \varnothing, z)\\
     &   & \raisebox{6pt}{11} & \raisebox{6pt}{12} & \raisebox{6pt}{13} &
\end{array}
\]

As the referee observed, the final values of the variables $\lop(x)$ record the order of last occurrence of the corresponding letters in the word $v$; for instance, the final row in the above example immediately tells us that the last occurrence word of $x^3yxyz^4xyz$ is $xyz$. Therefore, in order to verify that the last occurrence words of two given words $w$ and $w'$ are equal, it suffices to verify that when the above algorithm is applied to $w$ and $w'$, the final values of the variables $\lop(x)$ are the same. Similarly, the first nonzero values of the variables $\lop(x)$ record the order of first occurrence of the corresponding letters in $v$; in the above example, these values for $x$, $y$, and $z$ are respectively 1, 4, and 7, whence the first occurrence word is also $xyz$. Thus, in order to check that the first occurrence words of $w$ and $w'$ are equal, it suffices to check that the first non-zero values of the variables $\lop(x)$ appear in the same order. These observations show that the earlier separate algorithms for computing the first and the last occurrence words are in fact redundant. Nevertheless, we have retained them as they are conceptually very simple and have lower complexity.

\section{Conclusion}
\label{sec:applications}

\subsection{Future work} Obviously, the next natural step in studying identities of Kauffman monoids is to characterize the identities of $\mathcal{K}_n$ for $n>3$. Recently, two of the present authors (see~\citep{KV19}) have found a description of the identities of $\mathcal{K}_4$. It turns out that $\mathcal{K}_4$ satisfies precisely the same identities $\mathcal{K}_3$, which is a sort of surprise. The proof of this result is quite involved and relies on a geometric representation of Kauffman monoids rather than their presentation via generators and relations.

One can ask whether or not the coincidence of the identities of $\mathcal{K}_3$ and $\mathcal{K}_4$ extends further, say, to the identities of the monoid $\mathcal{K}_5$. The answer is negative: for instance, the identity $x^2yx\bumpeq xyx^2$, which holds in $\mathcal{K}_3$ (and hence, in $\mathcal{K}_4$) by Theorem~\ref{thm:description}, does not hold in $\mathcal{K}_5$, as the next proposition shows.
\begin{proposition}
\label{prop:K5}
If a homomorphism $\varphi\colon X^*\to\mathcal{K}_5$ extends the map \[\begin{cases}x\mapsto h_1h_2h_3\\ y\mapsto h_4\end{cases},\] then $(x^2yx)\varphi\ne(xyx^2)\varphi$.
\end{proposition}

\begin{proof}
First observe that
\begin{align*}
(h_1h_2h_3)^2=h_1h_2h_3h_1h_2h_3&=h_1h_2h_1h_3h_2h_3 &&\text{by~\eqref{eq:TL1}}\\
                                &=h_1h_3 &&\text{by~\eqref{eq:TL2}}.
\end{align*}
Therefore,
\begin{align*}
(x^2yx)\varphi=(h_1h_2h_3)^2h_4h_1h_2h_3&=h_1h_3h_4h_1h_2h_3 &&\text{as $(h_1h_2h_3)^2=h_1h_3$}\\
                               &=h_1^2h_3h_2h_4h_3&&\text{by~\eqref{eq:TL1}}\\
                               &=ch_1h_3h_2h_4h_3&&\text{by~\eqref{eq:TL4}},
\end{align*}
while
\begin{align*}
(xyx^2)\varphi=h_1h_2h_3h_4(h_1h_2h_3)^2&=h_1h_2h_3h_4h_1h_3 &&\text{as $(h_1h_2h_3)^2=h_1h_3$}\\
                               &=h_1h_2h_1h_3h_4h_3&&\text{by~\eqref{eq:TL1}}\\
                               &=h_1h_3&&\text{by~\eqref{eq:TL2}}.
\end{align*}
Since the words $ch_1h_3h_2h_4h_3=ch_{[1]}h_{[3,2]}h_{[4,3]}$ and $h_1h_3=h_{[1]}h_{[3]}$ are in Jones's normal form, Lemma~\ref{lem:jones} implies that they represent different elements of $\mathcal{K}_5$.
\end{proof}

At the moment, we possess no characterization of the identities of the monoid $\mathcal{K}_n$ for any $n>4$.

\subsection{Clustering phenomenon} Here we discuss an unexpected phenomenon revealed by the studies of identities of `interesting' semigroups: it turns out that semigroups coming from different parts of mathematics and having seemingly different nature tend to cluster with respect to their identities. For instance, comparing Theorem~\ref{thm:description} with the results by \citet{SV17}, we observe that the Kauffman monoid $\mathcal{K}_3$ shares the identities with the monoid $\mathcal{S}_2^1$, where $\mathcal{S}_2=\langle e,f \mid e^2=e,\ f^2=f\rangle$ is the free product of two trivial semigroups. Recall that in the proof of the `only if' part of Theorem~\ref{thm:description}, we exhibited an embedding of the semigroup $\mathcal{S}_2$ into $\mathcal{K}_3$. Clearly, this embedding extends to an isomorphism between the monoid $\mathcal{S}_2^1$ and a certain submonoid of the monoid $\mathcal{K}_3$, and therefore, every identity of the latter holds in $\mathcal{S}_2^1$. However, the fact that every identity of the submonoid isomorphic to $\mathcal{S}_2^1$ must hold in the whole monoid $\mathcal{K}_3$ appears to be somewhat amazing. Comparing Theorem~\ref{thm:description} with the results by \citet{KR79}, one can also observe that the monoid $\mathcal{K}_3$ satisfies the same identities as the least monoid containing the semigroup of adjacency patterns of words that was introduced and studied in \citep{KR79}.

Yet another interesting example has been found by \citet{DJK18}: the bicyclic monoid $\mathcal{B}:=\langle a,b\mid ba=1\rangle$ shares the identities with the monoid $UT_2(\mathbb{T})$ of all upper triangular $2\times 2$-matrices over the tropical semiring\footnote{Recall that the tropical semiring $\mathbb{T}$ is formed by the real numbers augmented with the symbol $-\infty$ under the operations $a\oplus b:=\max\{a,b\}$ and $a\otimes b:=a+b$, for which $-\infty$ plays the role of zero: $a\oplus-\infty=-\infty\oplus a=a$ and $a\otimes-\infty=-\infty\otimes a=-\infty$. A square matrix over $\mathbb{T}$ is said to be upper triangular if its entries below the main diagonal all $-\infty$.}. Similarly to the situation discussed in the preceding paragraph, $\mathcal{B}$ can be shown to be isomorphic to a submonoid of $UT_2(\mathbb{T})$, cf. \citep[Corollary 4.2]{IM10}, whence every identity of the latter holds in $\mathcal{B}$.  Again, it was unexpected that every identity of the submonoid isomorphic to $\mathcal{B}$ extends to the whole monoid $UT_2(\mathbb{T})$. \citet{Sh15} has provided a family of further interesting examples of semigroups which satisfy the same identities as the bicyclic monoid.

We mention in passing that the same clustering phenomenon occurs in the realm of finite monoids. For quite a representative example, the reader may compare the results of the papers \citep{AVZ15,JF,Vo04}. Each of these papers  studies identities of certain finite monoids that belong to several natural series parameterized by positive integers: Straubing monoids, Catalan monoids, Kiselman monoids, gossip monoids, etc. These monoids (whose definitions we do not reproduce here) arise in the literature due to completely unrelated reasons and consist of elements of very different nature. Nevertheless, it turns out that the $n$-th monoids in each of the series satisfy the same identities!

\bigskip

\small

\noindent\textbf{Acknowledgements.} The authors are very much indebted to the anonymous referee for her/his valuable remarks, especially, for the observation that the algorithm computing the multiset of jumps can also be used to compute the first and the last occurrence words.

Yuzhu Chen, Xun Hu, Yanfeng Luo have been partially supported by the Natural Science Foundation of China (projects no.\ 10971086, 11371177). M. V. Volkov acknowledges support from the Ministry of Science and Higher Education of the Russian Federation, project no.\  1.580.2016, the Competitiveness Program of Ural Federal University, and from the Russian Foundation for Basic Research, project no.\ 17-01-00551.


\begin{thebibliography}{}
\frenchspacing
\bibitem[Almeida et al.(2009)]{AVG09}
Almeida, J., Volkov, M.V., and Goldberg, S.V. [2009]: \emph{Complexity of the identity checking problem for finite semigroups}, J. Math. Sci. \textbf{158}(5), 605--614.

\bibitem[Ashikhmin et al.(2015)]{AVZ15}
Ashikhmin, D.N., Volkov, M.V., and Zhang, Wen Ting [2015]: \emph{The finite basis problem for Kiselman monoids}, Demonstratio Mathematica \textbf{48}(4), 475--492.

\bibitem[Auinger(2012)]{Au12}
Auinger, K. [2012]: \emph{Krohn--Rhodes complexity of Brauer type semigroups}, Port. Math. \textbf{69}(4), 341--360.

\bibitem[Auinger(2014)]{Au14}
Auinger, K. [2014]: \emph{Pseudovarieties generated by Brauer-type monoids}, Forum Mathematicum \textbf{26}, 1--24.

\bibitem[Auinger et al.(2015)]{ACHLV15}
Auinger, K., Chen, Yuzhu, Hu, Xun, Luo, Yanfeng, and Volkov, M.V. [2015]: \emph{The finite basis problem for Kauffman monoids}, Algebra Universalis \textbf{74}(3--4), 333--350.

\bibitem[Auinger et al.(2012)]{ADV12}
Auinger, K., Dolinka, I., and Volkov, M.V. [2012]: \emph{Equational theories of semigroups with involution}, J. Algebra \textbf{369}, 203--225.

\bibitem[Bokut' and Lee(2005)]{BL05}
Bokut', L.A., and Lee, D.V. [2005]: \emph{A Gr\"obner--Shirshov basis for the Temperley--Lieb--Kauffman monoid}, Izvestija Ural'skogo Gosudarstvennogo Universiteta \textbf{36}, 47--66 [Russian].

\bibitem[Borisavljevi\'c et al.(2002)]{BDP02}
Borisavljevi\'c,  M., Do\v{s}en, K., and Petri\'c, Z. [2002]: \emph{Kauffman monoids}, J. Knot Theory Ramifications \textbf{11}, 127--143.

\bibitem[Brauer(1937)]{Br37}
Brauer, R. [1937]: \emph{On algebras which are connected with the semisimple continuous groups}, Ann.\ Math. \textbf{38}, 857--872.

\bibitem[Clifford and Preston(1961)]{CP61}
Clifford, A.H., and Preston, G.B. [1961]: \textsf{The Algebraic Theory of Semigroups. Vol.I}, Amer.\ Math.\ Soc., Providence, R.I.

\bibitem[Daviaud et al.(2018)]{DJK18}
Daviaud, L., Johnson, M., and Kambites, M. [2018]: \emph{Identities in upper triangular tropical matrix semigroups and the bicyclic monoid}, J. Algebra \textbf{501}, 503--525.

\bibitem[Dolinka and East(2017)]{DE17}
Dolinka, I., and East, J. [2017]: \emph{The idempotent-generated subsemigroup of the Kauffman monoid}, Glasg. Math. J., \textbf{59}(3), 673--683.

\bibitem[Dolinka and East(2018)]{DE18}
Dolinka, I., and East, J. [2018]: \emph{Twisted Brauer monoids}, Proc. Royal Soc. Edinburgh, Ser. A \textbf{148A}, 731--750.

\bibitem[Dolinka et al.(2015)]{Dea15}
Dolinka, I., East, J., Evangelou, A., FitzGerald, D., Ham, N.,  Hyde, J., and  Loughlin, N. [2015]: \emph{Enumeration of idempotents in diagram semigroups and algebras}, J. Comb. Theory, Ser. A \textbf{131}, 119--152.

\bibitem[Dolinka et al.(2017)]{DEG17}
Dolinka, I., East, J., and Gray, R. [2017]: \emph{Motzkin monoids and partial Brauer monoids}, J. Algebra \textbf{471}, 251--298.

\bibitem[East(2011a)]{Ea11a}
East, J. [2011a]: \emph{Generators and relations for partition monoids and algebras}, J. Algebra \textbf{339}, 1--26.

\bibitem[East(2011b)]{Ea11b}
East, J. [2011b]: \emph{On the singular part of the partition monoid}, Internat. J. Algebra Comput. \textbf{21}(1-2), 147--178.

\bibitem[East(2014a)]{Ea14a}
East, J. [2014a]: \emph{Partition monoids and embeddings in 2-generator regular $*$-semi\-groups}, Period. Math. Hungar. \textbf{69}(2), 211--221.

\bibitem[East(2014b)]{Ea14b}
East, J. [2014b]: \emph{Infinite partition monoids}, Internat. J. Algebra Comput. \textbf{24}(4), 429--460.

\bibitem[East and FitzGerald(2012)]{EF12}
East, J., and FitzGerald, D.G. [2012]: \emph{The semigroup generated by the idempotents of a partition monoid}, J. Algebra \textbf{372}, 108--133.

\bibitem[East and Gray(2017)]{EG17}
East, J., and Gray, R. [2017]:  \emph{Diagram monoids and Graham--Houghton graphs: Idempotents and generating sets of ideals}, J. Comb. Theory, Ser. A \textbf{146}, 63--128.

\bibitem[East et al.(2018)]{Eea18}
East, J., Mitchell, J.D., Ru{\v s}kuc, N., Torpey, M. [2018]: \emph{Congruence lattices of finite diagram monoids}. Adv. Math., \textbf{333}, 931--1003.

\bibitem[FitzGerald and Lau(2011)]{FL11}
FitzGerald, D.G., and Lau, K.W. [2011]: \emph{On the partition monoid and some related semigroups}, Bull. Aust. Math. Soc. \textbf{83}(2), 273--288.

\bibitem[Horv\'{a}th et al.(2007)]{HLMS}
Horv\'{a}th, G., Lawrence, J., M\'{e}rai, L., and Szab\'{o}, Cs. [2007]: \emph{The complexity of the equivalence problem for nonsolvable groups}, Bull. London Math. Soc. \textbf{39}(3), 433--438.

\bibitem[Izhakian and Margolis(2010)]{IM10}
Izhakian, Z., and Margolis, S. [2010]: \emph{Semigroup identities in the monoid of two-by-two tropical matrices}, Semigroup Forum \textbf{80}(2), 191--218.

\bibitem[Jackson and McKenzie(2006)]{JM06}
Jackson, M., and McKenzie, R. [2006]: \emph{Interpreting graph colorability in finite semigroups},  Internat. J. Algebra Comput. \textbf{16}(1), 119--140.

\bibitem[Johnson and Fenner(2019)]{JF}
Johnson, M., and Fenner, P. [2019]: \emph{Identities in unitriangular and gossip monoids}, Semigroup Forum \textbf{98}(2), 338--354.

\bibitem[Jones(1983)]{Jo83}
Jones, V.F.R. [1983]: \emph{Index for subfactors}, Inventiones Mathematicae \textbf{72}, 1--25.

\bibitem[Kauffman(1990)]{Ka90}
Kauffman, L. [1990]: \emph{An invariant of regular isotopy}, Trans. Amer. Math. Soc. \textbf{318}, 417--471.

\bibitem[Kim and Roush(1979)]{KR79}
Kim, K.H., and Roush, F. [1979]: \emph{The semigroup of adjacency patterns of words}, in: \textsf{Algebraic Theory of Semigroups}, North Holland, Amsterdam (Colloq. Math. Soc. J\'anos Bolyai \textbf{20}), 281--297.

\bibitem[Kisielewicz(2004)]{Ki04}
Kisielewicz, A. [2004]: \emph{Complexity of semigroup identity checking},  Internat. J. Algebra Comput. \textbf{14}(4), 455--464.

\bibitem[Kitov and Volkov(2019)]{KV19}
Kitov, N.V., and Volkov M.V. [2019]: \emph{Identities of the Kauffman monoid $\mathcal{K}_4$ and of the Jones monoid $\mathcal{J}_4$}, Lect. Notes Comput. Sci., accepted, see preprint at \url{https://arxiv.org/pdf/1910.09190.pdf}.

\bibitem[Kl{\'\i}ma(2009)]{Kl09}
Kl{\'\i}ma, O. [2009]:  \emph{Complexity issues of checking identities in finite monoids}, Semigroup Forum \textbf{79}(3), 435--444.

\bibitem[Kl{\'\i}ma(2012)]{Kl12}
Kl{\'\i}ma, O. [2012]:  \emph{Identity checking problem for transformation monoids}, Semigroup Forum \textbf{84}(3), 487--498.

\bibitem[Kudryavtseva et al.(2006)]{KMM06}
Kudryavtseva, G., Maltcev, V., and Mazorchuk, V. [2006]: \emph{$\mathcal{L}$- and $\mathcal{R}$-cross-sections in the Brauer semigroup}, Semigroup Forum \textbf{72}(2), 223--248.

\bibitem[Kudryavtseva and Mazorchuk(2006)]{KM06}
Kudryavtseva, G., and Mazorchuk, V. [2006]: \emph{On presentations of Brauer-type monoids}, Cent. Eur. J. Math. \textbf{4}(3), 413--434.

\bibitem[Kudryavtseva and Mazorchuk(2007)]{KM07}
Kudryavtseva, G., and Mazorchuk, V. [2007]: \emph{On conjugation in some transformation and Brauer-type semigroups}, Publ. Math. Debrecen \textbf{70}(1-2), 19--43.

\bibitem[Lau and FitzGerald(2006)]{LF06}
Lau, K.W., and FitzGerald, D.G. [2006]: \emph{Ideal structure of the Kauffman and related monoids}, Comm. Algebra \textbf{34}, 2617--2629.

\bibitem[Maltcev and Mazorchuk(2007)]{MM07}
Maltcev, V., and Mazorchuk V. [2007]: \emph{Presentation of the singular part of the Brauer monoid}, Math. Bohem. \textbf{132}(3), 297--323.

\bibitem[Mazorchuk(1998)]{Ma98}
Mazorchuk, V. [1998]: \emph{On the structure of Brauer semigroup and its partial analogue}, Problems in Algebra \textbf{13}, 29--45.

\bibitem[Mazorchuk(2002)]{Ma02}
Mazorchuk, V. [2002]: \emph{Endomorphisms of $\mathfrak{B}_n$, $P\mathfrak{B}_n$, and $\mathfrak{C}_n$}, Comm. Algebra \textbf{30}(7), 3489--3513.

\bibitem[Murski\v\i{}(1968)]{Mu68}
Murski\v\i{}, V.L. [1968]: \emph{Examples of varieties of semigroups}, Math. Notes \textbf{3}(6), 423--427.

\bibitem[Plescheva and V\'{e}rtesi(2006)]{PV06}
Plescheva, S.V., and V\'{e}rtesi, V. [2006]: \emph{Complexity of the identity checking problem in a $0$-simple semigroup}, Izvestija Ural'skogo Gosudarstvennogo Universiteta \textbf{43}, 72--102 [Russian].

\bibitem[Seif(2005)]{Se05}
Seif, S. [2005]: \emph{The Perkins semigroup has co-NP-complete term-equivalence problem}, Internat. J. Algebra Comput. \textbf{15}(2), 317--326.

\bibitem[Seif and Szab\'o(2006)]{SS06}
Seif, S., and Szab\'o, Cs. [2006]: \emph{Computational complexity of checking identities in $0$-simple semigroups and matrix semigroups over finite fields}, Semigroup Forum \textbf{72}(2), 207--222.

\bibitem[Shneerson(2015)]{Sh15}
Shneerson, L.M. [2015]: \emph{On growth, identities and free subsemigroups for inverse semigroups of deficiency one},  Internat. J. Algebra Comput. \textbf{25}(1-2), 233--258.

\bibitem[Shneerson and Volkov(2017)]{SV17}
Shneerson, L.M., and Volkov, M.V. [2017]: \emph{The identities of the free product of two trivial semigroups}, Semigroup Forum \textbf{95}(1), 245--250.

\bibitem[Steinruecken(2015)]{St15}
Steinruecken, Ch. [2015]: \emph{Compressing sets and multisets of sequences}, IEEE Trans. Information Theory \textbf{61}(3), 1485--1490.

\bibitem[Temperley and Lieb(1971)]{TL71}
Temperley, H.N.V., and Lieb, E.H. [1971]: \emph{Relations between the `percolation' and `colouring' problem and other graph-theoretical problems associated with regular planar lattices: Some exact results for the percolation problem}, Proc. Roy. Soc. London, Ser. A \textbf{322}, 251--280.

\bibitem[Volkov(2004)]{Vo04}
Volkov, M.V. [2004]:  \emph{Reflexive relations, extensive transformations and piecewise testable languages of a given height},  Internat. J. Algebra Comput. \textbf{14}(5-6), 817--827.
\end{thebibliography}
\end{document}